\documentclass[letterpaper,11pt]{amsart}
\usepackage{tikz, tikz-cd, stmaryrd}
\usetikzlibrary{arrows}
\usepackage[font=footnotesize,labelfont=bf]{caption}
\usepackage{blkarray}
\usepackage{enumitem}
\usepackage{latexsym,array,delarray,epsfig,setspace,cleveref, mathtools,amssymb,mathrsfs,bbold}
\usepackage{subfig}
\usepackage{caption}

\newtheorem{theorem}{Theorem}
\numberwithin{theorem}{section}
\newtheorem{proposition}[theorem]{Proposition}

\newtheorem{lemma}[theorem]{Lemma}

\newtheorem{question}[theorem]{Question}
\theoremstyle{definition}
\newtheorem{definition}[theorem]{Definition}
\newtheorem{remark}[theorem]{Remark}
\newtheorem{example}[theorem]{Example}

\newcommand{\Cc}{{\mathbb C}}

\newcommand{\Rr}{{\mathbb R}}
\newcommand{\Zz}{{\mathbb Z}}

\newcommand{\ct}{\mathcal{T}}
\newcommand{\cl}{\mathcal{L}}
\newcommand{\cx}{\mathcal{X}}

\renewcommand{\v}{\mathfrak{v}}
\newcommand{\w}{\mathfrak{w}}

\newcommand{\st}{~|~}					

\newcommand{\B}{\mathbb{B}}

\DeclareMathOperator{\Spec}{Spec}  
\DeclareMathOperator{\Proj}{Proj}

\DeclareMathOperator{\Hom}{Hom}
\DeclareMathOperator{\Cl}{Cl}
\DeclareMathOperator{\SL}{SL}

\title{A Fano compactification of the $\SL_2(\Cc)$ free group character variety}
\author{Joseph Cummings and Christopher Manon}

\begin{document}

\maketitle

\begin{abstract}
We show that a certain compactification $\mathfrak{X}_g$ of the $\SL_2(\Cc)$ free group character variety $\cx(F_g, \SL_2(\Cc))$ is Fano.  This compactification has been studied previously by the second author, and separately by Biswas, Lawton, and Ramras.  Part of the proof of this result involves the construction of a large family of integral reflexive polytopes.\end{abstract}

\tableofcontents

\section{Introduction}

Let $\SL_2(\Cc)$ denote the group of $2\times 2$ complex matrices with determinant 1, and let $F_g$ denote the free group on $g$ letters. The \emph{character variety} $\cx(F_g, \SL_2(\Cc))$ is the moduli space of representations $\rho: F_g \to \SL_2(\Cc)$. Character varieties emerge naturally as moduli of local systems on a punctured Riemann surface and as generalizations of Teichm\"uller space \cite{Goldman}, \cite{Simpson1}, \cite{Simpson2}, \cite{Vogt}.  In this paper we study a compactification $\mathfrak{X}_g$ of $\cx(F_g, \SL_2(\Cc))$, of a form constructed by the second author in \cite{Manon-Outer}, and essentially by Biswas, Lawton, and Ramras \cite{BLR}.  The following is our main theorem. 

\begin{theorem}\label{thm-main}
The boundary $\mathfrak{X}_g \setminus \cx(F_g, \SL_2(\Cc))$ is the union of $g$ irreducible divisors $D_i$ $1 \leq i \leq g$, and the anti-canonical class of $\mathfrak{X}_g$ is $3(\sum_{i = 1}^g D_i)$.  The divisor $\sum_{i =1}^g D_i$ is very ample, in particular $\mathfrak{X}_g$ is Fano. 
\end{theorem}

The proof of Theorem \ref{thm-main} is by \emph{toric degeneration}, in particular we show that $\mathfrak{X}_g$ is Fano by finding a degeneration to a Fano toric variety. For each choice of a certain type of graph $\Gamma$ with edge set $E(\Gamma)$ and a spanning tree $\ct$ (see Section \ref{sec-polytopes}), we show that there is a \emph{toric flat family} $\underline{\pi_\Gamma}: \overline{X_\Gamma} \to \Cc^{E(\Gamma)}$ with general fiber $\mathfrak{X}_g$ and special fiber a Fano toric variety $\overline{Y_\Gamma}$. For an introduction to toric flat families see \cite{Kaveh-Manon-TVB}.  The toric variety $\overline{Y_\Gamma}$ is the projective toric variety associated to a polytope $P(\Gamma, \ct) \subset \Rr^{E(\Gamma)}$.  The following is then a critical ingredient in the proof of Theorem \ref{thm-main}.

\begin{theorem}\label{thm-mainpolytope}
For each choice of $\Gamma$, $\ct$ as above, the polytope $P(\Gamma, \ct)$ is normal, and its third Minkowski sum $3P(\Gamma, \ct)$ is (an integral translate of) an integral reflexive polytope. 
\end{theorem}

A broader class of graphs $\Gamma$ than those used in Theorem \ref{thm-mainpolytope} can be used to find toric degenerations of $\mathfrak{X}_g$, however it is not always the case that the corresponding polytopes are normal or reflexive.  In Section \ref{sec-polytopes} we prove several combinatorial results along these lines. 

We mention a conjecture formulated by Simpson \cite{Simpson} that a relative character variety should have a log Calabi-Yau compactification. In \cite{Whang}, Whang verifies this for relative $\SL_2(\Cc)$ character varieties of punctured surfaces by constructing a compactification which, while different from $\mathfrak{X}_g$, has a similar construction. Both compactifications can be viewed as solutions to certain inequalities on the values of the valuations on the coordinate ring of $\mathcal{X}(F_g, \SL_2(\mathbb{C}))$ constructed in \cite{Manon-Outer}. In principle, the techniques in Section \ref{sec-degeneration} can also be used to construct a toric degeneration of Whang's compactification.  We expect that the elegant combinatorial constructions found by Whang in \cite[Section 6]{Whang} have a polyhedral explanation, and will likewise have analogues in our setting due to the reflexivity of $3P(\Gamma, \ct)$.

The compactification $\mathfrak{X}_g$ is from a family of such compactifications constructed by the second author in \cite{Manon-Outer}.  A similar construction is explored for all simple groups $G$ of adjoint type by Biswas, Lawton, and Ramras in \cite{BLR}, see also \cite[Sections 4 and 5]{Manon-Outer}.  Their construction relies on the use of the \emph{wonderful compactification} $G \subset \overline{G}$ in the role of the compactification $\SL_2(\Cc) \subset X$ from Section \ref{sec-compactification}.  Roughly speaking, by extending $\SL_2(\Cc)$ to $X$, we allow the values of the generators of $F_g$ to take values in the (equivalence classes of) singular matrices.  Similarly, Biswas, Lawton, and Ramras allow the generators of $F_g$ to take values in the wonderful compactification of a more general group $G$.  We expect that the techniques used in this paper can be generalized to show compactifications like those in \cite{BLR} are Fano.  

\begin{example}
The polytopes in Figure \ref{fig: genus 2 polytopes} correspond to the anti-canonical embedding of the special fiber $\overline{Y_\Gamma}$ for each genus 2 trivalent graph $\Gamma$.
\end{example}

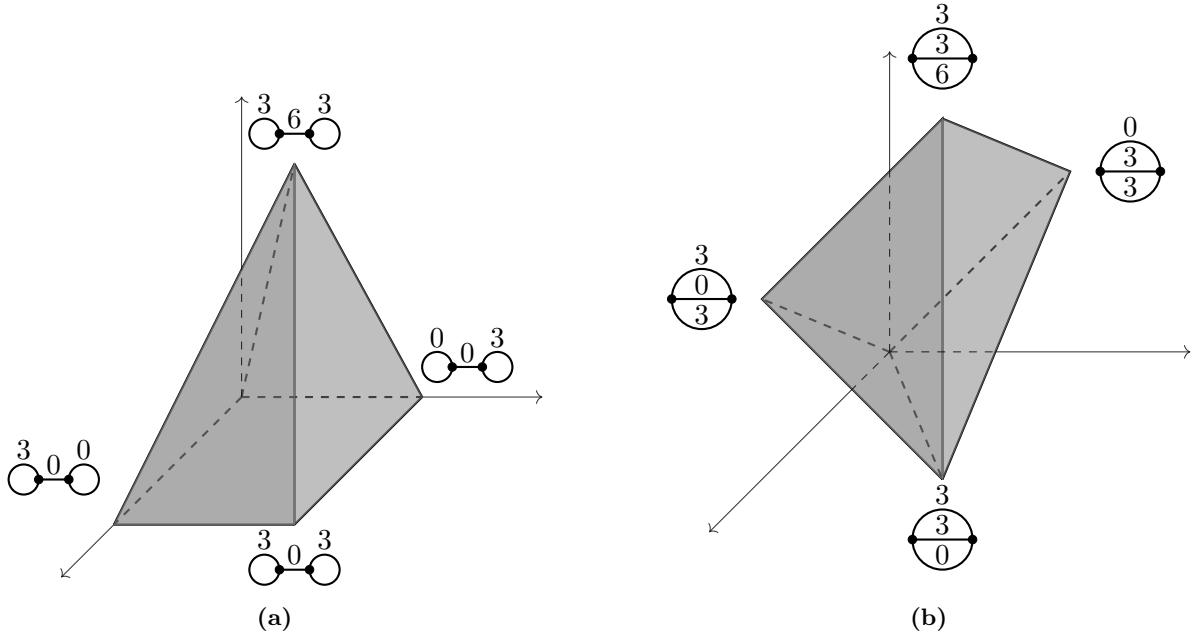
\begin{figure}
    \centering
    \subfloat[\label{subfig:dumbell polytope}]{
    \begin{tikzpicture}[scale = 0.8]
    \coordinate (V1) at (0,0);
        
    \coordinate (V2) at ({-3*sqrt(1/2)},{-3*sqrt(1/2)});
        \coordinate (A2) at ({-3*sqrt(1/2)-1.25},{-3*sqrt(1/2) + 0.75});
        \coordinate (B2) at ({-3*sqrt(1/2)-0.75},{-3*sqrt(1/2) + 0.75});
        \coordinate (MA2) at ({-3*sqrt(1/2)-1.5},{-3*sqrt(1/2) + 0.75});
        \coordinate (MB2) at ({-3*sqrt(1/2)-0.5},{-3*sqrt(1/2) + 0.75});
        \draw[thick] (A2) -- (B2);
        \draw[thick] (MA2) circle [radius = 0.25];
        \draw[thick] (MB2) circle [radius = 0.25];
        \draw[fill] (A2) circle [radius = 0.07];
        \draw[fill] (B2) circle [radius = 0.07];
        \node[label={90:$3$}] at (MA2) {};
        \node[label={90:$0$}] at ({-3*sqrt(1/2)-1},{-3*sqrt(1/2) + 0.5}) {};
        \node[label={90:$0$}] at (MB2) {};
        
    \coordinate (V3) at (3,0);
        \coordinate (A3) at (3.5,0.5);
        \coordinate (B3) at (4,0.5);
        \coordinate (MA3) at (3.25,0.5);
        \coordinate (MB3) at (4.25,0.5);
        \draw[thick] (A3) -- (B3);
        \draw[thick] (MA3) circle [radius = 0.25];
        \draw[thick] (MB3) circle [radius = 0.25];
        \draw[fill] (A3) circle [radius = 0.07];
        \draw[fill] (B3) circle [radius = 0.07];
        \node[label={90:$0$}] at (MA3) {};
        \node[label={90:$0$}] at (3.75,0.25) {};
        \node[label={90:$3$}] at (MB3) {};
        
    \coordinate (V4) at ({-3*sqrt(1/2) + 3}, {-3*sqrt(1/2)}); 
        \coordinate (A4) at ({-3*sqrt(1/2) + 2.75}, {-3*sqrt(1/2) - 0.75});
        \coordinate (B4) at ({-3*sqrt(1/2) + 3.25}, {-3*sqrt(1/2) - 0.75});
        \coordinate (MA4) at ({-3*sqrt(1/2) + 2.5}, {-3*sqrt(1/2) - 0.75}); 
        \coordinate (MB4) at ({-3*sqrt(1/2) + 3.5}, {-3*sqrt(1/2) - 0.75});
        \draw[thick] (A4) -- (B4);
        \draw[thick] (MA4) circle [radius = 0.25];
        \draw[thick] (MB4) circle [radius = 0.25];
        \draw[fill] (A4) circle [radius = 0.07];
        \draw[fill] (B4) circle [radius = 0.07];
        \node[label={90:$3$}] at (MA4) {};
        \node[label={90:$0$}] at ({-3*sqrt(1/2) + 3}, {-3*sqrt(1/2) - 1}) {};
        \node[label={90:$3$}] at (MB4) {};
        
    \coordinate (V5) at ({-3*sqrt(1/2) + 3},{-3*sqrt(1/2) + 6});
        \coordinate (A5) at ({-3*sqrt(1/2) + 2.75},{-3*sqrt(1/2) + 6.5});
        \coordinate (B5) at ({-3*sqrt(1/2) + 3.25},{-3*sqrt(1/2) + 6.5});
        \coordinate (MA5) at ({-3*sqrt(1/2) + 2.5},{-3*sqrt(1/2) + 6.5}); 
        \coordinate (MB5) at ({-3*sqrt(1/2) + 3.5},{-3*sqrt(1/2) + 6.5});
        \draw[thick] (A5) -- (B5);
        \draw[thick] (MA5) circle [radius = 0.25];
        \draw[thick] (MB5) circle [radius = 0.25];
        \draw[fill] (A5) circle [radius = 0.07];
        \draw[fill] (B5) circle [radius = 0.07];
        \node[label={90:$3$}] at (MA5) {};
        \node[label={90:$6$}] at ({-3*sqrt(1/2) + 3},{-3*sqrt(1/2) + 6.25}) {};
        \node[label={90:$3$}] at (MB5) {};

    \draw[thick, dashed] (V1) -- (V2);
    \draw[thick, dashed] (V1) -- (V3);
    \draw[thick] (V2) -- (V4) -- (V3);
    \draw[thick, dashed] (V1) -- (V5);
    \draw[thick] (V2) -- (V5);
    \draw[thick] (V3) -- (V5);
    \draw[thick] (V4) -- (V5);
    
    \draw[fill, color = gray, opacity = 0.65] (V2) -- (V4) -- (V5) -- (V2);
    \draw[fill, color = gray, opacity = 0.5] (V3) -- (V4) -- (V5) -- (V3);
    
    \draw[->, draw opacity = 0.8] (V2) -- (-3,-3);
    \draw[->, draw opacity = 0.8] (V3) -- (5,0);
    \draw[draw opacity = 0.8, dashed] (V1) -- (0,{3/sqrt(2)});
    \draw[->, draw opacity = 0.8] (0,{3/sqrt(2)}) -- (0,5);
    
\end{tikzpicture}
    } 
    \hfill \subfloat[\label{subfig:beachball polytope}]{
    \begin{tikzpicture}[scale = 0.8]
    \coordinate (V1) at (0,0);
    
    \coordinate (V2) at ({-3/sqrt(2) + 3}, {-3/sqrt(2)});
        \coordinate (A2) at ({-3/sqrt(2) + 2.5}, {-3/sqrt(2) - 1});
        \coordinate (B2) at ({-3/sqrt(2) + 3.5}, {-3/sqrt(2) - 1});
        \coordinate (M2) at ({-3/sqrt(2) + 3}, {-3/sqrt(2) - 1});
        \draw[thick] (M2) circle [radius = 0.5];
        \draw[thick] (A2) -- (B2);
        \draw[fill] (A2) circle [radius = 0.07];
        \draw[fill] (B2) circle [radius = 0.07];
        \node[label={90:$0$}] at ({-3/sqrt(2) + 3}, {-3/sqrt(2) - 1.75}) {};
        \node[label={90:$3$}] at ({-3/sqrt(2) + 3}, {-3/sqrt(2) - 0.75}) {};
        \node[label={90:$3$}] at ({-3/sqrt(2) + 3}, {-3/sqrt(2) - 1.25}) {};
        
    \coordinate (V3) at ({-3/sqrt(2)},{-3/sqrt(2) + 3});
        \coordinate (A3) at ({-3/sqrt(2)-1.5},{-3/sqrt(2) + 3});
        \coordinate (B3) at ({-3/sqrt(2)-0.5},{-3/sqrt(2) + 3});
        \coordinate (M3) at ({-3/sqrt(2)-1},{-3/sqrt(2) + 3});
        \draw[thick] (M3) circle [radius = 0.5];
        \draw[thick] (A3) -- (B3);
        \draw[fill] (A3) circle [radius = 0.07];
        \draw[fill] (B3) circle [radius = 0.07];
        \node[label={90:$3$}] at ({-3/sqrt(2)-1},{-3/sqrt(2) + 3.25}) {};
        \node[label={90:$0$}] at ({-3/sqrt(2)-1},{-3/sqrt(2) + 2.75}) {};
        \node[label={90:$3$}] at ({-3/sqrt(2)-1},{-3/sqrt(2) + 2.25}) {};
    
    \coordinate (V4) at (3,3);
        \coordinate (A4) at (3.5,3);
        \coordinate (B4) at (4.5,3);
        \coordinate (M4) at (4,3);
        \draw[thick] (M4) circle [radius = 0.5];
        \draw[thick] (A4) -- (B4);
        \draw[fill] (A4) circle [radius = 0.07];
        \draw[fill] (B4) circle [radius = 0.07];
        \node[label={90:$0$}] at (4,3.25) {};
        \node[label={90:$3$}] at (4,2.75) {};
        \node[label={90:$3$}] at (4,2.25) {};
        
    \coordinate (V5) at ({-3/sqrt(2) + 3}, {-3/sqrt(2) + 6});
        \coordinate (A5) at ({-3/sqrt(2) + 2.5}, {-3/sqrt(2) + 7});
        \coordinate (B5) at ({-3/sqrt(2) + 3.5}, {-3/sqrt(2) + 7});
        \coordinate (M5) at ({-3/sqrt(2) + 3}, {-3/sqrt(2) + 7});
        \draw[thick] (M5) circle [radius = 0.5];
        \draw[thick] (A5) -- (B5);
        \draw[fill] (A5) circle [radius = 0.07];
        \draw[fill] (B5) circle [radius = 0.07];
        \node[label={90:$3$}] at ({-3/sqrt(2) + 3}, {-3/sqrt(2) + 7.25}) {};
        \node[label={90:$3$}] at ({-3/sqrt(2) + 3}, {-3/sqrt(2) + 6.75}) {};
        \node[label={90:$6$}] at ({-3/sqrt(2) + 3}, {-3/sqrt(2) + 6.25}) {};
    
    \draw[thick, dashed] (V1) -- (V2);
    \draw[thick, dashed] (V1) -- (V3);
    \draw[thick, dashed] (V1) -- (V4);
    \draw[thick] (V2) -- (V5) -- (V3) -- (V2);
    \draw[thick] (V2) -- (V4) -- (V5);
    
    \draw[fill, color = gray, opacity = 0.65] (V2) -- (V3) -- (V5) -- (V2);
    \draw[fill, color = gray, opacity = 0.5] (V2) -- (V5) -- (V4) -- (V2);
    
    \draw[dashed, draw opacity = 0.8] (V1) -- ({-3/sqrt(2) + 3/2},{-3/sqrt(2) + 3/2});
    \draw[->, draw opacity = 0.8] ({-3/sqrt(2) + 3/2},{-3/sqrt(2) + 3/2}) -- (-3,-3);
    \draw[dashed, draw opacity = 0.8] (V1) -- ({3 - 3/(sqrt(2)+1)},0);
    \draw[->, draw opacity = 0.8] ({3 - 3/(sqrt(2)+1)},0) -- (5,0);
    \draw[draw opacity = 0.8, dashed] (V1) -- (0,3);
    \draw[->, draw opacity = 0.8] (0,3) -- (0,5);
\end{tikzpicture}
    } 
    \caption{The polytopes $3P(\Gamma,\ct)$ for the two genus 2 graphs. Each is reflexive after translating by $(-2,-2,-2)$. The vertices are labelled with the corresponding labelled graphs.}
    \label{fig: genus 2 polytopes}
\end{figure}

\tableofcontents

\section{The compactification $\mathfrak{X}_g$}\label{sec-compactification}

In \cite{Manon-Outer} a divisorial compactification $\cx(F_g, \SL_2(\Cc)) \subset \cx_\Gamma$ is constructed for every ribbon graph $\Gamma$ with first Betti number $g$. It can be shown that $\cx_\Gamma$ does not depend on the ribbon structure of $\Gamma$.  By \cite[Proposition 8.2]{Manon-Outer}, the boundary $\cx_\Gamma \setminus \cx(F_g, \SL_2(\Cc))$ is of normal crossings type, with one irreducible divisor $D_e$ for each edge $e \in E(\Gamma)$. These divisors are sufficient to give the class group of $\cx_\Gamma$.

\begin{proposition}\label{prop-class}
For any ribbon graph $\Gamma$ with first Betti number $g$, the class group is $$\Cl(\cx_\Gamma) \cong \bigoplus_{e \in E(\Gamma)} \Zz D_e.$$
\end{proposition}

\begin{proof}
The divisors $D_e$ are the irreducible components of the complement $\cx_\Gamma \setminus \cx(F_g, \SL_2(\Cc))$, so there is an exact sequence of groups:$$\Cc[\cx(F_g,\SL_2(\Cc))]^* \to \bigoplus_{e \in E(\Gamma)} \Zz D_e \to \Cl(\cx_\Gamma) \to \Cl(\cx(F_g,\SL_2(\Cc))) \to 0.$$ The character variety $\cx(F_g,\SL_2(\Cc))$ is factorial \cite{Lawton-Manon}, and the only units in the coordinate ring $\Cc[\cx(F_g,\SL_2(\Cc))]$ are constants.
\end{proof}

If the graph $\Gamma$ is trivalent, then it determines two cones $C_\Gamma$, $P(\Gamma)$ which give $\cx(F_g, \SL_2(\Cc))$ many of the features of an affine toric variety. Points $\gamma \in C_\Gamma \cong \Rr_{\geq 0}^{E(\Gamma)}$ each determine a real-valued valuation $\v_\gamma$ on $\Cc[\cx(F_g, \SL_2(\Cc))]$ called a length function.  This cone can be identified with a \emph{prime cone} in the tropical variety of $\cx(F_g, \SL_2(\Cc))$ associated to its embedding by a certain finite collection $Y_g$ of regular functions, \cite[Proposition 6.3]{Manon-Outer}.  We will use a compactification of the toric flat family associated to the integral points of $C_\Gamma$ for our main construction in Section \ref{sec-degeneration}.  

The second cone $P(\Gamma) \subset \Rr^{E(\Gamma)}$ is the set of $a \in \Rr^{E(\Gamma)}$ 
for which $$a(e) \leq a(f) + a(g),$$ $$a(f) \leq a(e) + a(g),$$ $$a(g) \leq a(e) + a(f),$$
for any three edges $e, f, g \in E(\Gamma)$ which contain a common vertex.  Accordingly, we say $a \in P(\Gamma)$ satisfies the \emph{triangle inequalities} at each vertex of $\Gamma$. There is a natural pairing $\langle -, - \rangle_\Gamma: C_\Gamma \times P(\Gamma) \to \Rr$ induced by the inner product on $\Rr^{E(\Gamma)}$. We often consider the cone $P(\Gamma)$ with respect to the sublattice $M_\Gamma \subset \Zz^{E(\Gamma)}$ defined by the condition $a(e) + a(f) + a(g) \in 2\Zz$ for every $e, f, g \in E(\Gamma)$ which share a common vertex. We let $N_\Gamma = \Hom(M_\Gamma, \Zz)$ denote the dual lattice. 

The trivalent graph $\Gamma$ also determines a basis $\B_\Gamma \subset \Cc[\cx(F_g, \SL_2(\Cc))]$ of the regular functions on $\cx(F_g, \SL_2(\Cc))$.  In particular, there is an element $\Phi_a \in \B_\Gamma$, called a spin diagram, for each lattice point $a \in P(\Gamma) \cap M_\Gamma$. For $a \in P(\Gamma) \cap M_\Gamma$ and $\gamma \in C_\Gamma$ we have $\v_\gamma(\Phi_a) = \langle \gamma, a \rangle_\Gamma$ by \cite[Corollary 6.12]{Manon-Outer}.  Moreover, by \cite[Proposition 8.3]{Manon-Outer}, the section space $H^0(\cx_\Gamma, \sum n_e D_e)$ is the subspace of $\Cc[\cx(F_g,\SL_2(\Cc)]$ spanned by those spin diagrams $\Phi_a$ with $a(e) \leq n_e$.  If $\hat{\Gamma}$ is a graph with first Betti number $g$, we can still obtain a nice basis of each section space of $\cx_{\hat{\Gamma}}$ by finding a trivalent graph $\Gamma$ and a sufficiently nice surjection $\pi: \Gamma \to \hat{\Gamma}$. 

\begin{proposition}\label{prop-sectionspaces}
Let $\Gamma$ and $\hat{\Gamma}$ be graphs with first Betti number equal to $g$, and suppose that $\Gamma$ is trivalent. Let $\pi: \Gamma \to \hat{\Gamma}$ be a map of graphs which collapses edges $f_1, \ldots, f_\ell$ while mapping the remaining edges bijectively onto the edges of $\hat{\Gamma}$, then the section space $H^0(\cx_{\hat{\Gamma}}, \sum_{e} n_e D_e)$ has a basis given by the elements of $\B_\Gamma$ with $a(\pi(e')) \leq n_{\pi(e')}$.
\end{proposition}

\begin{proof}
This follows from \cite[Proposition 8.3]{Manon-Outer}. 
\end{proof}

In the sequel we consider the case where $\hat{\Gamma}$ is a wedge of $g$ loops at a vertex, and $\Gamma$ is a trivalent graph with a choice of spanning tree $\ct \subset \Gamma$.  When $\ct$ is collapsed to a single vertex, the result is $\hat{\Gamma}$. We denote the compactification given by a wedge of $g$ loops by $\mathfrak{X}_g$.  Following \cite{Manon-Outer}, $\mathfrak{X}_g$ is constructed as a GIT quotient by an action of $\SL_2(\Cc)$ on a product $X^g$. Here $X$ is an $\SL_2(\Cc)\times \SL_2(\Cc)$ projective compactification of $\SL_2(\Cc)$, equivariantly embedded in $\mathbb{P}^4$ as the $0$-locus of the quadratic polynomial $A D - B C - t^2$.  We let $i: X \to \mathbb{P}^4$ be the inclusion map, and $D \subset X$ denote the divisor obtained by setting $t = 0$.  The sheaf of sections of $D$ is then the pullback of $\mathcal{O}(1)$ on $\mathbb{P}^4$, and can be identified with the sections of the pullback line bundle $\mathcal{L} = i^*\mathcal{O}(1)$. The space $\mathfrak{X}_g$ is the GIT quotient: $$\mathfrak{X}_g = \SL_2(\Cc) \backslash\!\backslash_{\mathcal{L}^{\boxtimes g}} X^g,$$
where $\SL_2(\Cc)$ acts through the diagonal of $\SL_2(\Cc)\times \SL_2(\Cc)$ on each component of $X^g$. The divisor $D \subset X$ can be regarded as a projectivization of the singular $2\times 2$ matrices. In particular, two singular matrices $A_1, A_2$ define the same point of $D$ if $A_1 = dA_2$, where $d \in \Cc^*$. In this way, $\mathfrak{X}_g$ is a space of $\SL_2(\Cc)$ representations of $F_g$ where the values of the generators of $F_g$ are allowed to degenerate to equivalence classes of singular matrices.

\section{The Polytopes}\label{sec-polytopes}

In this section we introduce the polytope $P(\Gamma, \ct)$ where $\Gamma$ is a trivalent graph with first Betti number $g$, and $\ct \subset \Gamma$ is a spanning tree. In Section \ref{sec-degeneration}, we argue that $3P(\Gamma,\ct)$ is the moment polytope for a toric degeneration of the compactification $\mathfrak{X}_g$. We let $E(\Gamma)$ and $V(\Gamma)$ denote the edge and vertex sets of $\Gamma$, and we let $\ell_1, \ldots, \ell_g$ denote the elements in the set $F(\Gamma, \ct) = E(\Gamma) \setminus E(\ct)$. 

\begin{definition}\label{def-polytope}
Let $P(\Gamma, \ct) \subset P(\Gamma) \subset \Rr^{E(\Gamma)}$ be the set of points $a$ such that $a(\ell) \leq 1$ for any $ \ell \in F(\Gamma, \ct)$. We refer to these conditions as the \emph{boundary inequalities}. 
\end{definition}

\begin{lemma}\label{lemma-isPolytope}
$P(\Gamma,\ct)$ is a polytope of dimension $3g-3$.
\end{lemma}

\begin{proof}
We begin by showing $P(\Gamma,\ct)$ is bounded. By definition, if $a \in P(\Gamma,\ct)$, then $0 \leq a(e) \leq 1$ for all $e \in F(\Gamma,\ct)$. Let $\ct_0 = \ct$ and let $L_0$ be the set of leaves of $\ct_0$. Each leaf in $\ct_0$ must be either adjacent to a loop or two edges in $F(\Gamma,\ct)$; as these edges are bounded by 1, we have that $0\leq a(e) \leq 2$.

For $i > 0$, set $\ct_i = \ct_{i-1} \setminus L_{i-1}$ and let $L_i$ be the set of leaves in $\ct_i$. We claim that $0 \leq a(e) \leq 2^{i+1}$ for all $e \in L_i$. By induction, we can assume that for every $e\in L_{i-1}$, we have $0 \leq a(e) \leq 2^{i+1}$. Suppose that $e \in L_i$. Then $e$ is adjacent to two edges $f_1,f_2 \in F(\Gamma,\ct) \cup L_0 \cup \dots \cup L_{i-1}$. The triangle inequalities imply that 
\[
    0 \leq a(e) \leq 2^{j+1} + 2^{k+1} \leq 2^{i+1}
\]
for some $-1\leq j,k < i$ finishing the inductive step. This process terminates after finitely many, say $n$, steps; hence, we see that $P(\Gamma,\ct) \subseteq [0,2^{n+1}]^{E(\Gamma)}$.

To see that $\dim P(\Gamma,\ct)$ is $3g-3$, we need to show it is full-dimensional as $|E(\Gamma)| = 3g-3$. To this end, note that $p = (2/3,\dots,2/3) \in \Rr^{E(\Gamma)}$ lies in the interior of $P(\Gamma,\ct)$ since it strictly satisfies all the boundary and triangle inequalities. Letting $r$ be the minimum distance from $p$ to any one of the supporting hyperplanes of $P(\Gamma,\ct)$, we see that $P(\Gamma,\ct)$ contains the ball of radius $r$ centered at $p$; hence, $P(\Gamma,\ct)$ is full-dimensional.
\end{proof}

Suppose $\Gamma$ is obtained by adding a loop at every leaf of $\ct$. In Section \ref{subsec-GF}, we show that the third Minkowski dilate of $P(\Gamma,\ct)$ is an integral translate of a reflexive polytope, and in Section \ref{subsec-normality}, we show that $P(\Gamma,\ct)$ is a normal lattice polytope. As we shall see, these results do not hold for general pairs $\Gamma,\ct$. Before presenting the proofs of these two results, we make a final remark about the lattices $M_\Gamma$ and $N_\Gamma$ when $\Gamma,\ct$ is of this prescribed form.

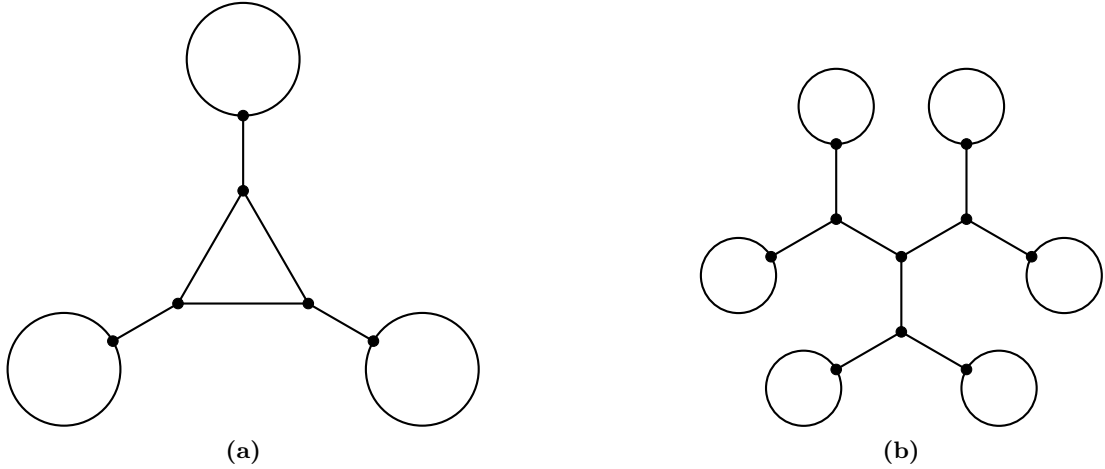
\begin{figure}
    \centering
    \subfloat[\label{fig:nontreeexample}]{
        \begin{tikzpicture}
    
        \coordinate (V1) at (0,1);
        \coordinate (V2) at ({cos(-30)},{sin(-30});
        \coordinate (V3) at ({cos(210)}, {sin(210)});
        \coordinate (A1) at (0,2);
        \coordinate (A2) at ({2*cos(-30)},{2*sin(-30});
        \coordinate (A3) at ({2*cos(210)}, {2*sin(210)});
    
        \draw [thick] (V1) -- (V2);
        \draw [thick] (V1) -- (V3);
        \draw [thick] (V2) -- (V3);
        \draw [thick] (V1) -- (A1);
        \draw [thick] (V2) -- (A2);
        \draw [thick] (V3) -- (A3);
        
        \draw [thick] (0,2.75) circle [radius = 0.75];
        \draw [thick] ({2.75*cos(-30)}, {2.75*sin(-30)}) circle [radius = 0.75];
        \draw [thick] ({2.75*cos(210)}, {2.75*sin(210)}) circle [radius = 0.75];
        
        \draw [fill] (V1) circle [radius = 0.07];
        \draw [fill] (V2) circle [radius = 0.07];
        \draw [fill] (V3) circle [radius = 0.07];
        \draw [fill] (A1) circle [radius = 0.07];
        \draw [fill] (A2) circle [radius = 0.07];
        \draw [fill] (A3) circle [radius = 0.07];
        
        \end{tikzpicture}
    }\hspace{1in}
    \subfloat[\label{fig:treewithloops}]{
        \centering
        \begin{tikzpicture}
    
        \coordinate (V1) at (0,0);
        \coordinate (V2) at ({cos(30)}, {sin(30)});
        \coordinate (V3) at ({cos(150)}, {sin(150)});
        \coordinate (V4) at (0,-1);
        \coordinate (V5) at ({cos(30)}, {sin(30) + 1});
        \coordinate (V6) at ({2*cos(30)}, 0);
        \coordinate (V7) at ({cos(150)}, {sin(150) + 1});
        \coordinate (V8) at ({cos(150) + cos(210)}, {sin(150) + sin(210)});
        \coordinate (V9) at ({cos(210)},{-1 + sin(210)});
        \coordinate (V10) at ({cos(-30)}, {-1 + sin(-30)});
        
        \coordinate (C1) at ({cos(30)}, {sin(30) + 1 + 0.5});
        \coordinate (C2) at ({2.5*cos(30)}, {0.5*sin(-30)});
        \coordinate (C3) at ({cos(150)}, {sin(150) + 1.5});
        \coordinate (C4) at ({cos(150) + 1.5*cos(210)}, {sin(150) + 1.5*sin(210)});
        \coordinate (C5) at ({1.5*cos(210)},{-1 + 1.5*sin(210)});
        \coordinate (C6) at ({1.5*cos(-30)}, {-1 + 1.5*sin(-30)});
        
        \draw [thick] (V1) -- (V2);
        \draw [thick] (V1) -- (V3);
        \draw [thick] (V1) -- (V4);
        \draw [thick] (V2) -- (V5);
        \draw [thick] (V2) -- (V6);
        \draw [thick] (V3) -- (V7);
        \draw [thick] (V3) -- (V8);
        \draw [thick] (V4) -- (V9);
        \draw [thick] (V4) -- (V10);
        
        \draw [thick] (C1) circle [radius = 0.5];
        \draw [thick] (C2) circle [radius = 0.5];
        \draw [thick] (C3) circle [radius = 0.5];
        \draw [thick] (C4) circle [radius = 0.5];
        \draw [thick] (C5) circle [radius = 0.5];
        \draw [thick] (C6) circle [radius = 0.5];
        
        \draw [fill] (V1) circle [radius = 0.07];
        \draw [fill] (V2) circle [radius = 0.07];
        \draw [fill] (V3) circle [radius = 0.07];
        \draw [fill] (V4) circle [radius = 0.07];
        \draw [fill] (V5) circle [radius = 0.07];
        \draw [fill] (V6) circle [radius = 0.07];
        \draw [fill] (V7) circle [radius = 0.07];
        \draw [fill] (V8) circle [radius = 0.07];
        \draw [fill] (V9) circle [radius = 0.07];
        \draw [fill] (V10) circle [radius = 0.07];
        \end{tikzpicture}
    } 
     \hfill
    \caption{Two examples of trivalent graphs. The one on the left is not one considered in Section \ref{subsec-GF} while the one the right is.}
    \label{fig:trivalent graph examples}
\end{figure}

\begin{lemma}\label{lem-lattice}
Let $\ct$ be a trivalent tree with $g$ leaves, and let $\Gamma$ be the graph obtained from $\ct$ by adding a loop to each leaf. The lattice $M_\Gamma \subseteq \Zz^{E(\Gamma)}$ is equal to 
\[
    \Zz^{F(\Gamma,\ct)} \oplus (2\Zz)^{E(\ct)}
\]
and $N_\Gamma$ is
\[
    \Zz^{F(\Gamma,\ct)} \oplus \left(\frac{1}{2}\Zz\right)^{E(\ct)}.
\]
\end{lemma}
\begin{proof}
Let $a \in M_\Gamma$. Set $\ct_0 = \ct$ and consider the leaves of $\ct_0$. Each leaf $e$ is connected to a loop $\ell$. Since $a \in M_\Gamma$, we have that $2a(\ell) + a(e) \in 2\Zz$. It follows that $a(\ell) \in \Zz$ and $a(e) \in 2\Zz$. 

For $i>0$, let $\ct_i$ be the subtree of $\ct_{i-1}$ obtained by removing all the leaves of $\ct_i$. Assume inductively that for each leaf $e$ of $\ct_{i-1}$, we have that $a(e) \in 2\Zz$. Consider a leaf $f$ of $\ct_i$. Since $\ct$ is trivalent, it is adjacent to edges $e_1,e_2 \in E(\ct_{i-1})$. As $a \in M_\Gamma$, we have that $a(e_1) + a(e_2) + a(f) \in 2\Zz$. Since $a(e_1)$ and $a(e_2)$ are even, we must have that $a(f)$ is even as well. Continuing in this way until $\ct_j$ is empty, we see that $M_\Gamma$ is $\Zz^{F(\Gamma,\ct)} \oplus (2\Zz)^{E(\ct)}$. 

The second statement about $N_\Gamma$ follows since it is the set of all $b \in \Rr^{E(\Gamma)}$ for which $a\cdot b \in \Zz$ for all $a \in M_\Gamma$.
\end{proof}

\subsection{The Gorenstein-Fano property}\label{subsec-GF}

As currently defined, $P(\Gamma, \ct)$ has the origin as a vertex; thus, no dilate will be reflexive. To amend this issue, we consider a new polytope $Q(\Gamma, \ct)$ defined by the following inequalities:

\begin{definition}
Let $Q(\Gamma, \ct) \subset \Rr^{E(\Gamma)}$ be the set of points satisfying the following conditions:

\begin{enumerate}
    \item For any three edges $e, f, g$ which share a commone vertex we must have: 
    \begin{align*}
    \frac{w(e) + w(f) - w(g)}{2}  &\geq -1, \\
    \frac{w(e) - w(f) + w(g)}{2}  &\geq -1, \\
    \frac{-w(e) + w(f) + w(g)}{2} &\geq -1.
    \end{align*}
    \item For any edge $\ell \in F(\Gamma, \ct)$ we must have:
    \[-w(\ell) \geq -1\]
\end{enumerate}
\end{definition}

\begin{remark}
The $\frac{1}{2}$'s are introduced so that the inward pointing normal vectors of $Q(\Gamma, \ct)$ are primitive with respect to the lattice $N_\Gamma$. In other words, the polar dual $Q(\Gamma, \ct)^\circ$ is a lattice polytope with respect to $N_\Gamma$.
\end{remark}

\begin{lemma}
The translation of $Q(\Gamma, \ct)$ by the vector consisting of all 2's is the third dilate of $P(\Gamma, \ct)$.
\end{lemma}

\begin{proof}
Let $\mathbf{2} = (2,\dotsc,2)^T \in \Rr^{E(\Gamma)}$. Let $w \in Q(\Gamma,\ct)$. Then we see that $w + \mathbf{2}$ satisfies the following inequalities:
\begin{enumerate}
    \item For any edges $e, f, g$ which share a common vertex, we have
    $$ \frac{(w(e) + 2) + (w(f) + 2) - (w(g) + 2)}{2} \geq 0, $$
    $$ \frac{(w(e) + 2) - (w(f) + 2) + (w(g) + 2)}{2} \geq 0, $$
    $$ \frac{-(w(e) + 2) + (w(f) + 2) + (w(g) + 2)}{2} \geq 0. $$
    
    \item For any edge $\ell \in F(\Gamma, \ct)$, we have
    $$ -(w(\ell) + 2) \geq -3. $$
\end{enumerate}
It follows that $Q(\Gamma, \ct) + \mathbf{2} \subseteq 3 P(\Gamma, \ct)$. Similarly, if $v \in 3 P(\Gamma, \ct)$, then $v - \mathbf{2} \in Q(\Gamma, \ct)$; hence, we have equality.
\end{proof}

The main result of this section is that $Q(\Gamma,\ct)$ is reflexive when $F(\Gamma,\ct)$ is a collection of $g$ loops. The proof is by induction on the first Betti number of the graph. Thus, we begin by showing the dumbbell graph in Figure \ref{fig:dumbell} yields a reflexive polytope. For this, it is enough to argue that $Q(\Gamma_1,\ct)$ is a lattice polytope with respect to $M_{\Gamma_1}$.

\begin{figure}
    \centering
    \subfloat[\label{fig:dumbell}]{
        \begin{tikzpicture}
    
        \coordinate (V) at (-0.5,0);
        \coordinate (W) at (0.5,0);
    
        \draw [thick] (W) -- (V);
        \draw [thick] (-1,0) circle [radius = 0.5];
        \draw [thick] (1,0) circle [radius = 0.5];
        
        \draw [fill] (V) circle [radius = 0.07];
        \draw [fill] (W) circle [radius = 0.07];
        \end{tikzpicture}
    }\hspace{1in}
    \subfloat[\label{fig:beachball}]{
        \centering
        \begin{tikzpicture}
    
        \coordinate (V) at (-1,0);
        \coordinate (W) at (1,0);
        
        \draw [thick] (W) -- (V);
        \draw [thick] (0,0) circle [radius = 1];
        
        \draw [fill] (V) circle [radius = 0.07];
        \draw [fill] (W) circle [radius = 0.07];
        \end{tikzpicture}
    } 
     \hfill
    \caption{The two trivalent genus 2 graphs $\Gamma_1$ and $\Gamma_2$, respectively.}
    \label{fig:genus2graphs}
    
\end{figure}
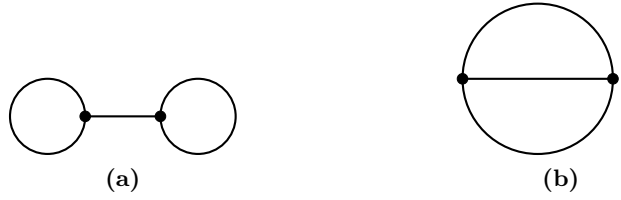

We denote the two loops by $\ell_1$ and $\ell_2$ and the bridge by $e$ and order the standard basis vectors of $\Rr^{E(\Gamma_1)}$ in this way. Then $Q(\Gamma_1, \ct)$ is the set of all $x \in \Rr^3$ satisfying the following inequalities
\[
     \begin{pmatrix}
        1 & 0 & -\frac{1}{2} \\
        0 & 1 & -\frac{1}{2} \\
        0 & 0 & -\frac{1}{2} \\
        -1 & 0 & 0 \\
        0 & -1 & 0
        \end{pmatrix} \begin{pmatrix} x_1 \\ x_2 \\ x_3 \end{pmatrix} \geq \begin{pmatrix} -1 \\ -1 \\ -1 \\ -1 \\ -1 \end{pmatrix}.
\]
The vertices of $Q(\Gamma_1, \ct)$ are the columns of the matrix below.
\[
    \begin{pmatrix}
    -2&1&-2&1&1\\
    -2&-2&1&1&1\\
    -2&-2&-2&-2&4
    \end{pmatrix}
\]
In particular, $Q(\Gamma_1, \ct)$ is a lattice polytope with respect $M_{\Gamma_1} = \Zz \ell_1 \oplus \Zz \ell_2 \oplus 2\Zz e$.

\begin{remark}
We also note that $\Gamma_2$ from Figure \ref{fig:beachball} yields a reflexive polytope. Label the edges $\ell_1,\ell_2,$ and $e$ arbitrarily and set $\ct = e$. Then $Q(\Gamma_2,\ct)$ is the set of all $x \in \Rr^3$ satisfying
\[
          \begin{pmatrix}
          \frac{1}{2} & \frac{1}{2} & -\frac{1}{2} \\
          \frac{1}{2} & -\frac{1}{2} & \frac{1}{2} \\
          -\frac{1}{2} & \frac{1}{2} & \frac{1}{2} \\
          -1 & 0 & 0 \\
          0 & -1 & 0 \\
          \end{pmatrix} 
          \begin{pmatrix}
          x_1 \\ x_2 \\ x_3
          \end{pmatrix} \geq
          \begin{pmatrix}
          -1 \\ -1 \\ -1 \\ -1 \\ -1
          \end{pmatrix}.
\]
The vertices of $Q(\Gamma_2,\ct)$ are the columns in the matrix below.
\[
    \begin{pmatrix}
    -2&1&1&-2&1\\
    -2&1&-2&1&1\\
    -2&-2&1&1&4
    \end{pmatrix}
\]
The column sums are all even; hence, each vertex is in $M_{\Gamma_2}$. 
\end{remark}

We need a lemma concerning certain faces of $Q(\Gamma,\ct)$.

\begin{lemma}\label{lem: faces}
Let $\Gamma$ be a trivalent genus $g\geq 3$ graph with $g$ loops, and let $\ct$ be its unique spanning tree. For $e \in E(\ct)$, then define 
\[
    F_e = 3P(\Gamma,\ct) \cap \{w \in \Rr^{E(\Gamma)} ~|~ w(e) = 0\}.
\]
$F_e$ is a face of $3P(\Gamma,\ct)$, and $F_e$ is isomorphic to 

\[
    3P(\Gamma_1,\ct_1) \times 3P(\Gamma_2,\ct_2).
\]
$\Gamma_1$ and $\Gamma_2$ are the connected components of the graph obtained by deleting $e$ contracting the pairs of edges adjacent to $e$; and $\ct_1$ and $\ct_2$ are their spanning trees, respectively. If $e$ is a leaf of $\ct$ and $\Gamma_1$ is just a single loop, then we take $P(\Gamma_1,\ct_1) = [0,1]$.
\end{lemma}

\begin{proof}
$F_e$ is a face since it is the intersection of $3P(\Gamma,\ct)$ with the two supporting hyperplanes $\{w ~|~ w(e) + w(f) - w(h) = 0\}$ and $\{w ~|~ w(e) - w(f) + w(h) = 0\}$ where $e,f,$ and $h$ are edges meeting at a vertex.

Suppose $e \in E(\ct)$ is a leaf adjacent to a loop $\ell$, and let $w \in F_e$. The triangle and boundary inequalities imply that $w(\ell) \in [0,3]$. The other pair of edges $f,h$ adjacent to $e$ must lie in $E(\ct)$ and since $w(e) = 0$, we see that $w(f) = w(h)$. Thus, any point $w$ in $F_e$ gives a point in $w' \in [0,3] \times 3P(\Gamma_2,\ct_2)$ by setting $w'(fh) = w(f) = w(h)$ and $w'(k) = w(k)$ for all other edges. On the other hand, a point $w' \in [0,3] \times 3P(\Gamma_2,\ct_2)$ yields a point in $w \in 3P(\Gamma, \ct)$ by setting $w(\ell)$ to be whatever was in the first component, $w(e) = 0$, $w(f) = w(h) = w'(fh)$, and $w(k) = w'(k)$ for all other edges. Thus, the claim holds when $e \in E(\ct)$ is a leaf.

Suppose $e \in E(\ct)$ is not a leaf, and it is adjacet to two pairs of edges $f,h$ and $f',h'$ all in $E(\ct)$. Again, we have that $w(f) = w(h)$ and $w(f') = w(h')$. The map which sends $w \in F_e$ to $w' \in 3P(\Gamma_1,\ct_1) \times 3P(\Gamma_2,\ct_2)$ defined by $w'(fh) = w(f) = w(h)$ and $w'(f'h') = w'(f') = w(h')$ and $w'(k) = w(k)$ for all other edges $k$ is well-defined and bijective, so the claim holds in this case as well.
\end{proof}

\begin{theorem}\label{thm:reflexive}
$Q(\Gamma,\ct)$ is reflexive whenever the pair $\Gamma, \ct$ is such that $F(\Gamma,\ct)$ is a collection of $g$ loops.
\end{theorem}

\begin{proof}
Let $\Gamma$ be a graph with first Betti number $g\geq 3$ consisting of a spanning tree $\ct$ and $g$ loops. It is enough to show that $Q(\Gamma, \ct) + \mathbf{2} = 3P(\Gamma, \ct)$ is a lattice polytope with respect to $M_\Gamma$. Recall that for any $w \in 3P(\Gamma,\ct)$ we have the triangle inequalities for each $e, f, h$ sharing a vertex:
\begin{align*}
    w(e) + w(f) - w(h) &\geq 0, \\
    w(e) - w(f) + w(h) &\geq 0, \\
    -w(e) + w(f) + w(h) &\geq 0,
\end{align*}
and we have boundary inequalities for each $\ell \in F(\Gamma,\ct)$:
\[ w(\ell) \leq 3. \]

Let $a$ be a vertex of $3P(\Gamma,\ct)$. The dimension of $3P(\Gamma,\ct)$ is $3g-3$; thus, $a$ is the intersection of at least $3g-3$ supporting hyperplanes. First, suppose there are three edges $e, f, h$ meeting at a vertex so that at least two of the three triangle inequalities are equalities on $a$, i.e.
\[
    a(e) + a(f) = a(h) \text{ and } a(e) + a(h) = a(f).
\]
This with the third triangle inequality imply that $a(e) = 0$ and $a(f) = a(h)$. We can assume that $e \in E(\ct)$, because if say $e=f$ were a loop and $a(e) = 0$, then we would have that $a(h) = 0$ with $h\in E(\ct)$, and we could work with $h$ instead. Moreover, if $f'$ and $h'$ are the (possibly not distinct) edges which are also adjacent to $e$, then the facts that $a(e)=0$ and the triangle inequalities for $e,f',h'$ imply that $a(f') = a(h')$.

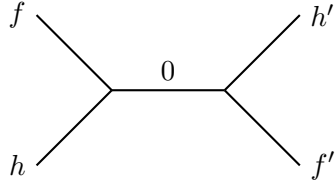
\begin{center}
    \begin{tikzpicture}[scale = 1, thick]
        \draw (0,0) -- (1,1);
        \draw (1,1) -- (0,2);
        \draw (1,1) -- (2.5,1);
        \draw (2.5,1) -- (3.5,2);
        \draw (2.5,1) -- (3.5,0);
            
        \draw (1.75,1) node[above]{$0$};
            
        \draw (0,0) node[left]{$h$};
        \draw (0,2) node[left]{$f$};
        \draw (3.5,2) node[right]{$h'$};
        \draw (3.5,0) node[right]{$f'$};
    \end{tikzpicture}
\end{center}

Now, consider the graph $\Gamma'$ obtained by deleting $e$ and concatenating the pairs of edges $f,h$ and $f',h'$. Note that $\Gamma'$ has two connected components both of which are trees with loops attached at each leaf, and their first Betti numbers are strictly smaller than $g$.

Let $a'$ be the induced labelling on $\Gamma'$ via the map from Lemma \ref{lem: faces}. By Lemma \ref{lem: faces}, we have that $a'$ is a vertex of $3P(\Gamma_1,\ct_1) \times 3P(\Gamma_2,\ct_2)$. By the induction hypothesis, this is a lattice polytope with respect to $M_{\Gamma_1} \times M_{\Gamma_2}$ (if $\Gamma_1$ is a loop then we take $M_{\Gamma_1}$ to be $\Zz$). Thus, in order to see that $a \in M_\Gamma$, we only need to check that $a(e) + a(f) + a(h) \in 2\Zz$ and $a(e) + a(f') + a(h') \in 2\Zz$. This follows since $a'(fh),a'(f'h') \in \Zz$ and since $a'(fh) = a(f) = a(h)$ and $a'(f'h') = a(f') = a(h')$.

Next, we consider the case when $a$ is the intersection of at least $3g-3$ supporting hyperplanes but for each triple of edges $e, f, h$ sharing a common vertex at most one of the triangle inequalities is an equality on $a$. In other words, there is no $e$ so that $a(e) = 0$. Since $\Gamma$ has $2g-2$ vertices, there are two possibilities:
\begin{enumerate}
    \item there is one triangle equality for every triple $e, f, h$ meeting at a vertex, and $a(\ell) = 3$ for all but possibly one $\ell \in F(\Gamma, \ct)$, or
    \item $a(\ell) = 3$ for all $\ell \in F(\Gamma, \ct)$, and there is a single triangle equality for all but exactly one triple $e, f, h$ meeting at a vertex,
\end{enumerate}

In case (1), we can see that for any edge $a(e)$ is either the sum or difference of $a(f)$ and $a(h)$ where $e,f,h$ meet at a vertex. Start at a leaf of $\ct$. This leaf is adjacent to a loop in $F(\Gamma, \ct)$. Since the genus is at least 3, we can assume that this edge satisfies $a(\ell) = 3$. Using the triangle equality at this node, we see the value on the leaf is 6 which is even. Continuing in this fashion, we can solve for values on all edges in $e$ in $\Gamma$, and we see that $a \in \Zz^{E(\Gamma)}$, and in fact, we can see that $a(e) \in 2\Zz$ for all edges in $\ct$. To see $a \in M_\Gamma$, consider a triple of edges sharing a vertex $e,f,h$. Without loss of generality, we can assume $a(e) = a(f) + a(h)$, so
\[ a(e) + a(f) + a(h) = 2a(e) \in 2\Zz. \]
It follows that $a \in M_\Gamma$.

In case (2), the same argument works to show that $a \in \Zz^{E(\Gamma)}$ and that $a(e) + a(f) + a(h) \in 2\Zz$ for every triple meeting at a vertex except one. Call this triple $e',f',h'$. We must show that $a(e') + a(f') + a(h') \in 2\Zz$. This follows since these are either all tree edges in which case, they are all even, or $e' = f'$ is a loop and $h'$ is a leaf in $\ct$ in which case $6 + a(h') \in 2\Zz$ as $a(h')$ is even.
\end{proof}

We end this section with a discussion about the hypotheses in Theorem \ref{thm:reflexive}. As shown earlier, the graph from Figure \ref{fig:beachball} does not satisfy the conditions of the theorem; however, $Q(\Gamma_2,\ct)$ is still reflexive. As we explain in the following example, $Q(K_4,\ct)$ is reflexive if and only if $\ct$ is trivalent.  

\begin{example}\label{ex: K4}
Consider the pair $K_4, \ct_1$ from Figure \ref{fig:reflexivek4}, so $\ct_1$ is the set of solid edges in \ref{fig:reflexivek4}. We label the edges in $F(K_4,\ct_1)$ $\ell_1,\ell_2,\ell_3$ clockwise starting at the top edge, and we label the tree edges $e_4,e_5,e_6$ starting at the edge adjacent to $\ell_1$ and $\ell_2$ and proceeding clockwise. Then $Q(K_4,\ct_1)$ is the defined by the four sets of triangle inequalities as well as the boundary inequalities $w(\ell_i) \leq 1$ for $i = 1,2,3$. The vertices of $Q(K_4, \ct_1)$ can be computed and are the columns in the matrix below. The rows (from top to bottom) correspond to the edges $\ell_1,\ell_2,\ell_3,e_4,e_5,$ and $e_6$, respectively.
\[
    \left(\begin{smallmatrix}
    -2 & 1 &-2 & 1 & 1 & 1& -2 & 1 & -2 & 1 & -2 & 1 & 1 & 1 & 1\\
    -2 & 1 & 1 & -2 & 1 & -2 & 1 & 1 & -2 & 1 & 1 & 1 & -2 & 1 & 1\\
    -2 & 1 & -2 & 1 & 1 & -2 & 1 & -2 & 1 & -2 & 1 & 1 & 1 & 1 & 1\\
    -2 & -2 & 1 & 1 & 4 & 1 & 1 & -2 & -2 & 4& 1 & 4 & 1 & -2 & 4\\
    -2 & -2 & 1 & 1 & 4 & -2 & -2 & 1 & 1 & 1 & 4 & -2 & 1 & 4 & 4\\
    -2 & -2 & -2 & -2 &-2 & 1 & 1 & 1 & 1 & 1 & 1 & 4 & 4 & 4 & 4
    \end{smallmatrix}\right)
\]
One can check that each vertex lies in $M_{K_4}$ by multiplying the matrix above by the matrix $M$ below and checking that all entries are indeed even.
\[
    M := \left(\begin{smallmatrix}
    1 & 1 & 0 & 1 & 0 & 0 \\
    0 & 1 & 1 & 0 & 1 & 0 \\
    1 & 0 & 1 & 0 & 0 & 1 \\
    0 & 0 & 0 & 1 & 1 & 1
    \end{smallmatrix}\right)
\]

Now, consider $Q(K_4, \ct_2)$ where $\ct_2$ comprises of the edges $\ell_2, e_4,$ and $e_6$. The polytope is defined in a similar way: it has the same twelve triangle inequalities, but the boundary inequalities are $w(\ell_1) \leq 1$, $w(\ell_3) \leq 1$, and $w(e_5) \leq 1$. The vertices of $Q(K_4, \ct_2)$ are the columns in the matrix below.
\[
    \left(\begin{smallmatrix}
    -2&1&-2&1&1&1&-2&1&-2&\color{red}{1}&1&-2&1&1&1&1\\
    -2&1&1&-2&4&-2&1&1&-2&\color{red}{1}&1&4&1&-2&4&4\\
    -2&1&-2&1&1&-2&1&-2&1&\color{red}{1}&-2&1&1&1&1&1\\
    -2&-2&1&1&1&1&1&-2&-2&\color{red}{-2}&4&4&4&1&1&7\\
    -2&-2&1&1&1&-2&-2&1&1&\color{red}{1}&1&1&-2&1&1&1\\
    -2&-2&-2&-2&-2&1&1&1&1&\color{red}{1}&1&1&4&4&4&4
    \end{smallmatrix}\right)
\]
Using the matrix $M$, we can check that all the vertices except the one colored in red is in $M_{K_4}$. Indeed, the following computation shows this vertex colored in red is not in $M_{K_4}$.
\[
    M
    \left(\begin{smallmatrix}
    1 \\ 1 \\ 1 \\ -2 \\ 1 \\ 1
    \end{smallmatrix}\right)
    = \left(\begin{smallmatrix}
    0 \\ 3 \\ 3 \\ 0
    \end{smallmatrix}\right)
\]

\end{example}

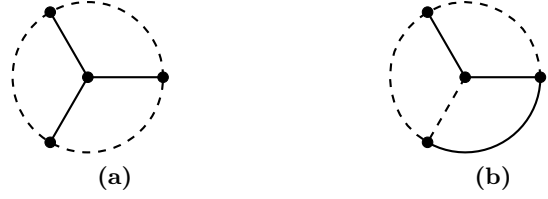
\begin{figure}
    \centering
    \subfloat[\label{fig:reflexivek4}]{
        \begin{tikzpicture}
    
        \coordinate (V1) at (0,0);
        \coordinate (V2) at (1,0);
        \coordinate (V3) at ({cos(120)}, {sin(120)});
        \coordinate (V4) at ({cos(240)}, {sin(240)});
    
        \draw [thick] (V1) -- (V2);
        \draw [thick] (V1) -- (V3);
        \draw [thick] (V1) -- (V4);
        \draw [thick, dashed] (V1) circle [radius = 1];
        
        \draw [fill] (V1) circle [radius = 0.07];
        \draw [fill] (V2) circle [radius = 0.07];
        \draw [fill] (V3) circle [radius = 0.07];
        \draw [fill] (V4) circle [radius = 0.07];
        
        \end{tikzpicture}
    }\hspace{1in}
    \subfloat[\label{fig:notreflexivek4}]{
        \centering
        \begin{tikzpicture}
    
        \coordinate (V1) at (0,0);
        \coordinate (V2) at (1,0);
        \coordinate (V3) at ({cos(120)}, {sin(120)});
        \coordinate (V4) at ({cos(240)}, {sin(240)});
    
        \draw [thick] (V1) -- (V2);
        \draw [thick] (V1) -- (V3);
        \draw [dashed, thick] (V1) -- (V4);
        
        \draw [dashed, thick, domain = 0:240] (V1) plot ({cos(\x)}, {sin(\x)});
        \draw [thick, domain = 240:360] (V1) plot ({cos(\x)}, {sin(\x)});
        
        \draw [fill] (V1) circle [radius = 0.07];
        \draw [fill] (V2) circle [radius = 0.07];
        \draw [fill] (V3) circle [radius = 0.07];
        \draw [fill] (V4) circle [radius = 0.07];
        \end{tikzpicture}
    } 
     \hfill
    \caption{Up to isomorphism, the pair $K_4, \ct$ is one of the two graphs above. In each the dashed edges are the edges in $F(K_4,\ct)$ and the solid edges are $\ct$.}
    \label{fig:K4differentSpanningTrees}
\end{figure}

As Example \ref{ex: K4} suggests, deciding whether $Q(\Gamma,\ct)$ is reflexive or not is very sensitive with respect to both $\Gamma$ and choice of spanning tree $\ct$. After some reflection, one can see that the inductive step in the proof of Lemma \ref{lem: faces} fails for general pairs $\Gamma, \ct$. Indeed, when contracting edges, you may contract an edge $\ell$ that lies in $F(\Gamma,\ct)$ with a tree edge $e$. Then the boundary condition on $\ell$ imposes an extra condition, so one is left not with $Q(\Gamma_1,\ct_1) \times Q(\Gamma_2,\ct_2)$ but rather this polytope with an extra boundary condition on an edge in the tree. This does not happen for the class of graphs considered in Lemma \ref{lem: faces} since we are only every contracting tree edges. As of yet, we are unsure what pairs $\Gamma,\ct$ produce a reflexive polytope $Q(\Gamma,\ct)$, and so we ask the following question.

\begin{question}\label{question: when reflexive?}
For what pairs $\Gamma,\ct$ is $Q(\Gamma,\ct)$ reflexive?
\end{question}

We can answer Question \ref{question: when reflexive?} when the genus is 2 or 3, and we have an obstruction on $\Gamma,\ct$ for $Q(\Gamma,\ct)$ being reflexive when $g \geq 4$.

\begin{proposition}
Both trivalent graphs with first Betti number 2 yield reflexive polytopes. If $\Gamma$ has first Betti number 3, then $Q(\Gamma,\ct)$ is reflexive if and only if $\Gamma,\ct$ is one of the graphs in Figure \ref{fig: reflexive genus 3}.
\end{proposition}

\begin{figure}
    \subfloat[\label{subfig:rattle}]{
    \begin{tikzpicture}
        \coordinate (V1) at (-0.5,0);
        \coordinate (V2) at (0.5,0);
        \coordinate (V3) at (1.25,0.75);
        \coordinate (V4) at (1.25, -0.75);
    
        \draw [fill] (V1) circle [radius = 0.07];
        \draw [fill] (V2) circle [radius = 0.07];
        \draw [fill] (V3) circle [radius = 0.07];
        \draw [fill] (V4) circle [radius = 0.07];
    
        \draw[thick] (V1) -- (V2);
        \draw[thick,dashed] (V3) -- (V4);
        \draw[thick, dashed] (-1,0) circle [radius = 0.5];
        \draw [thick, domain = 90:270] (V1) plot ({0.75*cos(\x) + 1.25}, {0.75*sin(\x)});
        \draw [thick, dashed, domain = -90:90] (V1) plot ({0.75*cos(\x) + 1.25}, {0.75*sin(\x)});
    \end{tikzpicture}
    } \hfill \subfloat[\label{subfig:treewithloops}]{
    \begin{tikzpicture}
        \coordinate (V1) at (-1,0);
        \coordinate (V2) at (1,0);
        \coordinate (V3) at (0,1);
        \coordinate (V4) at (0,0);
    
        \draw [fill] (V1) circle [radius = 0.07];
        \draw [fill] (V2) circle [radius = 0.07];
        \draw [fill] (V3) circle [radius = 0.07];
        \draw [fill] (V4) circle [radius = 0.07];
    
        \draw[thick] (V1) -- (V2);
        \draw[thick] (V3) -- (V4);
        \draw[thick,dashed] (-1.5,0) circle [radius = 0.5];
        \draw[thick,dashed] (1.5,0) circle [radius = 0.5];
        \draw[thick,dashed] (0,1.5) circle [radius = 0.5];
    \end{tikzpicture}
    } \hfill \subfloat[\label{subfig:K4}]{
    \begin{tikzpicture}
        \coordinate (V1) at (0,0);
        \coordinate (V2) at (1,0);
        \coordinate (V3) at ({cos(120)}, {sin(120)});
        \coordinate (V4) at ({cos(240)}, {sin(240)});
    
        \draw [thick] (V1) -- (V2);
        \draw [thick] (V1) -- (V3);
        \draw [thick] (V1) -- (V4);
        \draw [thick, dashed] (V1) circle [radius = 1];
        
        \draw [fill] (V1) circle [radius = 0.07];
        \draw [fill] (V2) circle [radius = 0.07];
        \draw [fill] (V3) circle [radius = 0.07];
        \draw [fill] (V4) circle [radius = 0.07];
    \end{tikzpicture}
    }
    \caption{Up to isomorphism, these are the three pairs $\Gamma,\ct$ of genus 3 graphs where $Q(\Gamma,\ct)$ is reflexive. The edges in each spanning tree are marked with solid edges, while the edges in $F(\Gamma,\ct)$ are dashed.}
    \label{fig: reflexive genus 3}
\end{figure}
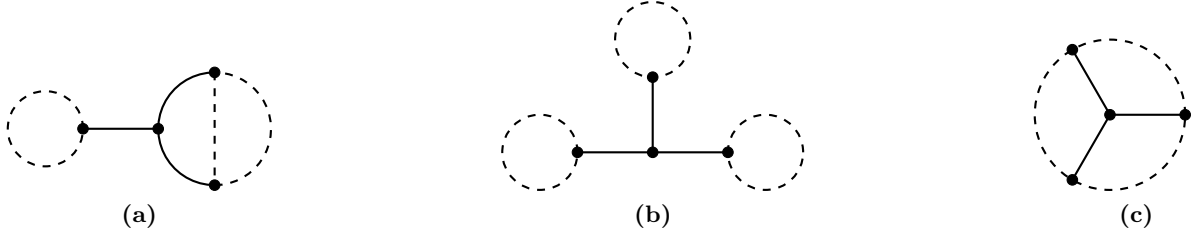

\begin{proposition}\label{prop: reflexive obstruction}
Let $\Gamma$ be a trivalent graph of with first Betti number $g$ with $k$ loops where $2 \leq k < g$. Let $\ct$ be a spanning tree of $\Gamma$, and consider a non-loop edge $f \in F(\Gamma,\ct)$ which is not adjacent to any other edges in $F(\Gamma,\ct)$. Under these conditions, the polytope $Q(\Gamma,\ct)$ is not reflexive.
\end{proposition}

\begin{proof}
We will show that $Q(\Gamma,\ct) + \mathbf{2}$ is not a lattice polytope. Fix a spanning tree $\ct$, a non-loop edge $f \in F(\Gamma,\ct)$, and two loops $\ell_1,\ell_2$. There is a unique path $P$ starting at the vertex of $\ell_1$ and ending at the vertex of $\ell_2$ passing through $f$ where all other edges are in $\ct$. Since $\ell_1$ and $\ell_2$ are adjacent to bridges in $\Gamma$, we know that the path is of length at least 5. Suppose the edges in this path are $\{f_1,f_2,\dotsc,f_i\}$, so some $f_j = f$ and every other $f_k \in \ct$. We also know that $f_1$ is the bridge connected to $\ell_1$ and $f_i$ is the bridge connected to $\ell_2$. Let $v \in \Rr^{E(\Gamma)}$ be the point below
\[
    v(e) = \begin{cases}
        3 &\text{if }e = f_k \text{ for } 1\leq k \leq i \\
        \frac{3}{2} &\text{if }e = \ell_1,\ell_2 \\
        0 & \text{otherwise}
    \end{cases}.
\]
To see that $v$ is a vertex, note that $v$ is the unique solution to the following set of equations and inequalities.
\begin{enumerate}
    \item For each triple meeting at a vertex $f_j, f_{j+1}, e$ where $0 \leq j \leq i-1$, we have
    \begin{align*}
        -x_{f_j} + x_{f_{j+1}} + x_e &= 0, \\
        x_{f_j} - x_{f_{j+1}} + x_e &= 0, \\
        x_{f_j} + x_{f_{j+1}} - x_e &\geq 0,
    \end{align*}
    and we have that 
    \begin{align*}
        2x_{\ell} - x_{f_i} &= 0.
    \end{align*}
    \item For every other triple $e,f,h$ where $e,f,h \notin P$, we have
    \begin{align*}
        x_e + x_f - x_h &= 0, \\
        x_e - x_f + x_h &= 0, \\
        -x_e + x_f + x_h &= 0.
    \end{align*}
    \item For every $e \in F(\Gamma,\ct)$, we have
    \[
        x_e \leq 3,
    \]
    and for $f_j$, we have
    \begin{align*}
        x_{f_j} &= 3.
    \end{align*}
\end{enumerate}
The equalities in (1) and (3) ensure that $v$ is as prescribed on $P$, and the equations in (2) and the inequalities in (1) ensure that $v(e) = 0$ away from $P$. It follows that $v$ is a vertex of $Q(\Gamma, \ct) + \mathbf{2}$; thus, $Q(\Gamma,\ct)$ is not a lattice polytope.
\end{proof}

\begin{remark}
This is not the only obstruction on $\Gamma,\ct$ to $Q(\Gamma, \ct)$ being reflexive. For example, we have checked $Q(\Gamma,\ct)$ is not reflexive for the complete bipartitie graph $K_{3,3}$ and for the Petersen graph with various choices of spanning tree in each, yet these graphs are simple.
\end{remark}

\subsection{Normalilty}\label{subsec-normality}

Now we show that $P(\Gamma, \ct)$ is a normal lattice polytope with respect to the lattice $M_\Gamma$ when $\Gamma$ is constructed by adding $g$ loops at the leaves of $\ct$. 

\begin{definition}
Let $CP(\Gamma, \ct) \subset \Rr^{E(\Gamma)}\times \Rr$ be the cone over the polytope $P(\Gamma, \ct)\times \{1\}.$ Let $S(\Gamma, \ct)$ be the affine semigroup $CP(\Gamma, \ct) \cap (M_\Gamma \times \Zz)$.
\end{definition}

\begin{theorem}\label{thm-normality}
Let $\Gamma$ be obtained from a trivalent tree $\ct$ by adding loops at the leaves of $\ct$, then the affine semigroup $S(\Gamma, \ct)$ is generated by the lattice points $(P(\Gamma, \ct) \times \{1\}) \cap (M_\Gamma\times \Zz)$.  In particular, $P(\Gamma, \ct)$ is a normal lattice polytope. 
\end{theorem}

The proof of Theorem \ref{thm-normality} relies on the normality of a closely-related polytope $\Delta(\ct)$. 

\begin{definition}
Let $\ct$ be a trivalent tree, and let $\Delta(\ct) \subset \Rr^{E(\ct)}$ be the set of points $w$ satisfying the following conditions:
\begin{enumerate}
    \item For any three edges $e, f, g \in E(\ct)$ which share a common vertex we must have: $$w(e) \leq w(f) + w(g),$$ $$w(f) \leq w(e) + w(g),$$ $$w(g) \leq w(e) + w(f).$$
    \item For any edge $e$ containing a leaf of $\ct$ we must have: $$w(e)\leq 2.$$
\end{enumerate}

Let $M_\ct \subset \Zz^{E(\ct)} \subset \Rr^{E(\ct)}$ denote the sublattice of even integer tuples.  
\end{definition}

It may be more efficient to define $\Delta(\ct)$ with respect to the usual integer lattice, and make the second inequality of the definition $w(e) \leq 1$. We choose to present $\Delta(\ct)$ in the manner above to make a connection with a class of polyhedra which play a prominent role in the study of the moduli of weighted configurations on the projective line \cite{HMSV}, and the moduli of rank $2$ parabolic vector bundles on a pointed projective line \cite{Manon-ACB}, \cite{Manon-Trees}.

\begin{lemma}\label{lem-deltanormal}
The polytope $\Delta(\ct)$ is a normal lattice polytope with respect to the lattice $M_\ct$.
\end{lemma}

\begin{proof}
Let $\ct'$ be the trivalent tree obtained by adding two new edges at every leaf of $\ct$. The graded affine semigroup consisting of the lattice points of the Minkowski sums $k\Delta(\ct)$, $k \geq 1$ can be seen to be the semigroup $S_{\ct'}(1, \ldots, 1)$ from \cite[Definition 1.2]{Manon-Trees}. The lemma is a consequence of \cite[Theorem 1.8]{Manon-Trees}.  
\end{proof}

To relate $P(\Gamma, \ct)$ to $\Delta(\ct)$ we imagine splitting each loop in $\Gamma$ into two edges. The result is the tree $\ct'$ from the proof of Lemma \ref{lem-deltanormal}. This construction does not result in an isomorphism between $P(\Gamma, \ct)$ and $\Delta(\ct)$, however it does allow us to use the normality of $\Delta(\ct)$ to decompose a degree $L$ element of $S(\Gamma, \ct)$ into $L$ degree $1$ elements.  

\begin{proposition}\label{prop-IDP}
Let $w \in \left(L\cdot P(\Gamma, \ct)\times \{L\}\right) \cap \left(M_\Gamma\times \Zz\right)$ be a degree $L$
element of $S(\Gamma, \ct)$, then there are $L$ lattice points $w_1, \ldots, w_L \in \left(P(\Gamma, \ct)\times \{1\}\right) \cap \left(M_\Gamma\times \Zz\right)$ such that $w_1 + \cdots + w_L = w$.
\end{proposition}

\begin{proof}
First we observe that the point $w \in L\cdot P(\Gamma, \ct)$ gives a point
$u \in L\cdot \Delta(\ct)$ by restricting $w: E(\Gamma) \to \Zz$ to $E(\ct) \subset E(\Gamma)$. If $e \in E(\ct)$ shares a vertex with a loop $\ell \in F(\Gamma, \ct)$, then $w(e) \leq 2w(\ell) \leq 2L$.  Moreover, the defining parity condition $w(e) + w(f) + w(g) \in 2\Zz$ of $M_\Gamma$ implies that $w(e) \in 2\Zz$.  As a consequence, $w$ weights all edges of $\ct$ with even integers.   

Next we use Lemma \ref{lem-deltanormal} to find $u_1, \ldots, u_L \in \Delta(\ct) \cap M_\ct$ such that $u_1 + \cdots + u_L = u$.  We must show that each $u_i$ extends to a $w_i \in P(\Gamma, \ct) \cap M_\Gamma$ in such a way that $w_1 + \cdots + w_L = w$. It suffices to show that this extension can be carried out at each loop of $\Gamma$. As above, let $e \in E(\ct)$ be an edge which shares a vertex with a loop $\ell \in F(\Gamma, \ct)$.  For each $i$, $u_i(e)$ is either $0$ or $2$, and there are precisely $N = \frac{1}{2}w(e) \leq w(\ell)$ weightings $u_{i_1}, \ldots, u_{i_N}$ where $u_{i_j}(e) = 2$. We extend each of these weightings to $\Gamma$ be letting $w_{i_j}(\ell) = 1$. To ensure that $w(\ell) = w_1(\ell) + \cdots + w_L(\ell)$ it remains to set $w_k(\ell) = 1$ for some set of size $w(\ell) - \frac{1}{2}w(e)$ such that $w_k(e) = 0$. As $N = \frac{1}{2}w(e) \leq w(\ell)$, this is always possible.  After making such a choice at each leaf $\ell$ we obtain $w_1, \ldots, w_L \in P(\Gamma, L) \cap M_\Gamma$ which sum to $w$. 
\end{proof}

Theorem \ref{thm-normality} is a consequence of Proposition \ref{prop-IDP}.

\section{The degeneration of $\mathfrak{X}_g$}\label{sec-degeneration}

The purpose of this section is to construct a flat degeneration of $\mathfrak{X}_g$ to a Fano projective toric variety.  Let $\Gamma$ be a trivalent graph of genus $g$, and fix a spanning tree $\ct$ of $\Gamma$. Recall that $E(\Gamma)$ denotes the set of edges of $\Gamma$, and $F(\Gamma, \ct)$ is the set of edges of $\Gamma$ not in $\ct$. 

\subsection{A piecewise-linear valuation on $\cx(F_g, \SL_2(\Cc))$}

Let $A_g$ be the coordinate ring of the free group $\SL_2(\Cc)$-character variety, and recall the basis $\B_\Gamma$ of spin diagrams $\Phi_a$ associated to the integral points $a \in P(\Gamma)\cap M_\Gamma$.  In the terminology of \cite{Kaveh-Manon-NOK} and \cite{Kaveh-Manon-TVB}, $\B_\Gamma$ is a \emph{linear adapted basis} for the valuations $\v_\gamma$, $\gamma \in C_\Gamma$. The following summarizes this property, for a proof see \cite[]{Manon-Outer}. 

\begin{proposition}\label{prop-adapted}
For any trivalent graph $\Gamma$ and $\gamma \in C_\Gamma$ we have:

\begin{enumerate}
    \item $\v_\gamma(\Phi_a) = \langle \gamma, a \rangle,$
    \item $\v_\gamma(\sum c_a \Phi_a) = \max\{\v_\gamma(\Phi_a) \mid c_a \neq 0\}.$
\end{enumerate}

\end{proposition}

There is a distinguished point of $C_\Gamma$ for each edge $e \in E(\Gamma)$, namely the function which assigns $0$ to every element of $E(\Gamma) \setminus \{e\}$ and assigns $1$ to $e$. In the sequel we abuse notation by refering to this point as $e$. 

The data of the valuations $\v_\gamma$ and basis $\B_\Gamma$ can be viewed as a valuation $\v_\Gamma$ on $A_g$ which takes values in the semifield $\mathcal{O}_{C_\Gamma}$ of piecewise-linear functions on the cone $C_\Gamma$, as described in \cite{Kaveh-Manon-TVB}. The basis $\B_\Gamma \subset A_g$ is then a so-called linear adapted basis of $\v_\Gamma: A_g \to \mathcal{O}_{C_\Gamma}$.  The following is a consequence of \cite[Proposition 6.1]{Kaveh-Manon-NOK} and \cite[Theorem 1.2]{Kaveh-Manon-TVB}. Let $Y_\Gamma$ denote the normal affine toric variety associated to the cone $P(\Gamma)$ and the lattice $M_\Gamma$. 

\begin{proposition}\label{prop-flataffinefamily}
Let $\Gamma$ be a trivalent graph with first Betti number $g$. There is a toric flat family $$\pi_\Gamma: X_\Gamma \to \Cc^{E(\Gamma)}.$$  Any fiber of this family over a point in $(\Cc^*)^{E(\Gamma)}$ is isomorphic to $\cx(F_g, \SL_2(\Cc))$.  The fiber over the origin of $\Cc^{E(\Gamma)}$ is isomorphic to $Y_\Gamma$. 
\end{proposition}

We construct a degeneration of $\mathfrak{X}_g$ for every graph $\Gamma$ with first Betti number $g$ along with choice of spanning tree $\ct$ by fiber-wise compactifying the family $X_\Gamma$. This yields a projective flat family of the form 
\[
    \underline{\pi_\Gamma} : \overline{X_\Gamma} \to \Cc^{E(\Gamma)}
\]
where each fiber is a projective variety containing the original fiber as a dense open subset. Our construction is more or less a relative version of the compactifications described in \cite[\S 6]{Kaveh-Manon-NOK}.

\subsection{Compactifying the family}

Let $\Gamma,\ct$ be as in the previous section. We construct a new ring $R_g$ which is a graded subring of $A_g[t_e ~|~ e \in F(\Gamma, \ct)]$ by letting $$(R_g)_b = \{f \mid \v_\Gamma(f)[e] \leq b(e)\} \subset A_gt^b.$$

\begin{proposition}\label{prop-cox}
For any graph $\Gamma$ with first Betti number $g$ and spanning tree $\ct$, $R_g$ is the Cox ring of $\mathfrak{X}_g$: $$R_g \cong Cox(\mathfrak{X}_g).$$
\end{proposition}

\begin{proof}
We use Proposition \ref{prop-sectionspaces} on the graph $\Gamma$ and the image graph $\hat{\Gamma}$ obtained by contracting $\ct$ to a single vertex.  The graded component $(R_g)_b \subset A_g$ is then seen to be the section space $H^0(\mathfrak{X}_g, \sum_{e \in E(\hat{\Gamma})} b(e) D_e)$. By Proposition \ref{prop-class}, the summands of $R_g$ run over all elements of $\Cl(\mathfrak{X}_g)$. 
\end{proof}

We consider an affine toric flat family $\hat{\pi}_\Gamma: \hat{X}_\Gamma \to \Cc^{E(\Gamma)}$.  A general fiber of this family is isomorphic to $\Spec(R_g)$, and the special fiber is a toric degeneration of $\Spec(R_g)$. To make this construction we require a piecewise-linear valuation $\w_\Gamma: R_g \to \mathcal{O}_{C_\Gamma}$.  We obtain $\w_\Gamma$ by extending $\v_\Gamma$ to the Laurent polynomial ring $A_g[t_e \mid e \in F(\Gamma, \ct)]$, and then considering the corresponding induced valuation on $R_g$. The next proposition summarizes consequences of this construction. 

\begin{proposition}\label{prop-coxflatfamily}
There is a piecewise-linear valuation $\w_\Gamma: R_g \to \mathcal{O}_{C_\Gamma}$ with adapted basis $\{\Phi_at^b \mid a \in P(\Gamma) \cap M_\Gamma, b \in \Zz^{F(\Gamma, \ct)}, a(e) \leq b(e) \forall e\in F(\Gamma, \ct)\}$.  The special fiber of the corresponding toric flat family $\hat{\pi}_\Gamma: \hat{X}_\Gamma \to \Cc^{E(\Gamma)}$ is the normal affine toric variety associated to the cone $\hat{P}(\Gamma) = \{(a, b) \mid a(e) \leq b(e) \forall e \in F(\Gamma, \ct)\} \subset P(\Gamma) \times \Rr^g$.  The action of $(\Cc^*)^{F(\Gamma, \ct)}$ on $\Spec(R_g)$ corresponding to the class group grading extends to a fiberwise action on the whole family $\hat{X}_\Gamma$. 
\end{proposition}

Observe that Proposition \ref{prop-coxflatfamily} implies that the Cox ring $R_g$ is both finitely generated and Cohen-Macaulay. 

Now consider the character $\chi : (\Cc^*)^{F(\Gamma, \ct)} \to \Cc^*$ given by $\chi((t_\ell)_{\ell\in F(\Gamma, \ct)}) = \prod_{\ell \in F(\Gamma, \ct)} t_\ell$. Our desired compactification is then given by the GIT quotient of this family with respect to $\chi$.
\[
    \underline{\pi_\Gamma} : \overline{X_\Gamma} = \widehat{X}_\Gamma \sslash_\chi (\Cc^*)^{F(\Gamma, \ct)} \to \Cc^{E(\Gamma)}
\]

\begin{proposition}
\label{prop:family is proj}
$\underline{\pi_\Gamma}$ is a projective morphism.

\end{proposition}

\begin{proof}
We let $\Cc[s_f\mid f \in E(\Gamma)]$ denote the coordinate ring of $\Cc^{E(\Gamma)}$. 
The space $\hat{X}_\Gamma$ is the $\Proj$ of a graded algebra $S_\chi = \bigoplus_{N \geq 0} \mathbf{F}_N$, where $\mathbf{F}_N = \{\Phi_a t^b s^c \mid a(e) \leq N, e \in F(\Gamma, \ct), a(f) \leq c(f) f \in E(\Gamma)\}$.  This is a summand of the Rees algebra of $\w_\Gamma$, and so defines an affine toric flat family over $\Cc^{E(\Gamma)}$. In particular, $S_\chi$ is a flat $\Cc[s_f \mid f \in E(\Gamma)]$ algebra. Specializing to the origin of $\Cc^{E(\Gamma)}$ yields the graded affine semigroup algebra $\Cc[S(\Gamma, \ct)]$, which is generated by its degree $1$ component. The conclusion of \cite[Proposition 4.6]{Kaveh-Manon-TVB} now implies that $\mathbf{F}_1$ generates $S_\chi$ as a $\Cc[s_f \mid f\in E(\Gamma)]$ algebra. Accordingly, we can define an embedding of the family $\overline{X_\Gamma}$ into a projective space over $\Cc^{E(\Gamma)}$. 

\end{proof}

\begin{proposition}
\label{prop:family is flat}
$\underline{\pi_\Gamma} : \overline{X_\Gamma} \to \Cc^{E(\Gamma)}$ is a flat family and $\overline{X_\Gamma}$ is Cohen-Macaulay.
\end{proposition}

\begin{proof}
First, we show that $\overline{X_\Gamma}$ is Cohen-Macualay. The special fiber of the family $\hat{X}_\Gamma$ is a saturated affine semigroup algebra, so it is Cohen-Macaulay.  It follows by the Hochster-Roberts theorem that $\Proj(\bigoplus_{N\geq 0} \mathbf{F}_N)$ is Cohen-Macaulay as well. We also know that the fibers of $\underline{\pi_\Gamma}$ are all of dimension $3g-3 = |E(\Gamma)|$ by \cite[Lemma 6.8]{Kaveh-Manon-TVB}, and we know the base $\Cc^{E(\Gamma)}$ is smooth. By Hironaka's criterion, we see that the family is flat.
\end{proof}

\begin{remark}
The family $\overline{X_\Gamma}$, the compactification $\mathfrak{X}_g$, and the polytope $P_\Gamma$ are constructed by considering the spin diagrams $\Phi_a$ with graph $\Gamma$ which satisfy $\max_{e \in F(\Gamma, \ct)}\{a(e)\} \leq 1$. The compactification considered in \cite{Whang} corresponds to the condition $\sum_{e \in F(\Gamma, \ct)} a(e) \leq 1$.  Accordingly, the toric degeneration techniques used in this section can be used on the compactification studied in \cite{Whang} as well.  It would be interesting to compute the corresponding polytope in that case. 
\end{remark}

\subsection{The Special Fiber}

In this section, we study the special fiber of $\underline{\pi_\Gamma} : \overline{X_\Gamma} \to \Cc^{E(\Gamma)}$. In particular, we show that the anti-canonical divisor of $\overline{Y_\Gamma}$ corresponds to the polytope $Q(\Gamma,\ct)$; thus, $\overline{Y_\Gamma}$ is Gorenstein-Fano when $\Gamma$ is a tree $\ct$ with loops attached to the edges. Let $M_\Gamma$ and $P(\Gamma)$ be as in the previous section, and let $P^\vee(\Gamma)$ be the dual cone of $P(\Gamma)$.

\begin{lemma}
$P^\vee(\Gamma)$ is a full-dimensional strongly convex rational polyhedral cone in $N_\Gamma \otimes \Rr$.
\end{lemma}

The valuation $\w_\Gamma$ has associated graded
\[
    \mathrm{gr}_{\w_\Gamma} (R_g) = \bigoplus_{(a,b)\in \hat{S}_\Gamma} \Cc t^{(a,b)}
\]
where $\hat{S}_\Gamma = \{(a,b) \in (P(\Gamma) \cap M_\Gamma) \times \Zz^{F(\Gamma, \ct)} ~|~ a(e) \leq b(e) \text{ for all }e \in F(\Gamma, \ct)\} = \hat{P}_\Gamma \cap M_\Gamma \times \Zz^{F(\Gamma, \ct)}$. We can describe the rays of the dual cone $\hat{P}^\vee_\Gamma \subset (N_\Gamma \times \Zz^{F(\Gamma, \ct)})_\Rr$. Let \(d_1, \dotsc, d_g\) be the standard basis vectors of \(\Rr^{F(\Gamma, \ct)}\) and let $(b_f)_{f\in E(\Gamma)}$ be the standard basis vectors of \(\Rr^{E(\Gamma)}\). Then, we have that 
\[
\hat{P}^\vee_\Gamma(1) = \{\Rr_{\geq 0}(d_\ell - b_\ell) \st \ell \in \cl \} \cup \{\Rr_{\geq 0}(0, v) \st v\in P^\vee_\Gamma(1)\}.
\]

Set $U_\sigma = \Spec(\Cc[\hat{S}_\Gamma])$ be the corresponding affine toric variety. This is the special fiber of the family $\widehat{X}_\Gamma \to \Cc^{E(\Gamma)}$, and the special fiber of $\overline{X_\Gamma} \to \Cc^{E(\Gamma)}$ is the GIT quotient of $U_\sigma$ by $(\Cc^*)^{F(\Gamma, \ct)}$ with respect to the character $\chi$ from the previous section.

\begin{theorem}
\label{thm-special fiber is GF}
The special fiber of $\overline{X_\Gamma} \to \Cc^{E(\Gamma)}$, denoted $\overline{Y_\Gamma}$, is a normal projective toric variety. The polytope associated to the anti-canonical divisor on $\overline{Y_\Gamma}$ is $Q(\Gamma,\ct)$; hence, $\overline{Y_\Gamma}$ is Gorenstein Fano when $\Gamma$ is a tree $\ct$ with loops attached to each leaf.
\end{theorem}

\begin{proof}
Let $\Sigma_{\Gamma,\ct}$ be the defining fan of the toric variety $\overline{Y_\Gamma}$. Let $(b_e)_{e\in E(\Gamma)}$ be the standard basis for $\Rr^{E(\Gamma)}$. For each vertex $v \in V(\Gamma)$, there are three triangle inequalities, and these correspond to three torus-invariant divisors $D^{++-}_v,D^{+-+}_v,$ and $D^{-++}_v$.  These divisors correspond to the ray generators
\[
    \frac{b_e + b_f - b_h}{2}, \frac{b_e - b_f + b_h}{2}, \frac{-b_e + b_f + b_h}{2} \in N_\Gamma
\]
of $\Sigma_{\Gamma,\ct}$ where $e,f,h$ meet at $v$. Moreover, for each edge $\ell \in F(\Gamma,\ct)$, we have the torus-invariant divisor $D_\ell$ corresponding to the ray generator $-b_\ell \in N_\Gamma$. The anti-canonical divisor on $\overline{Y_\Gamma}$ is the sum these torus-invariant prime divisors:
\[
    -K_{\overline{Y_\Gamma}} = \sum_{\ell\in F(\Gamma,\ct)} D_\ell + \sum_{v\in V(\Gamma)} \left(D^{++-}_v + D^{+-+}_v + D^{-++}_v\right).
\]
The polytope corresponding to $-K_{\overline{Y_\Gamma}}$ is $Q(\Gamma,\ct)$.
\end{proof}

\subsection{Transferring the Fano Property}

Our last result pertains to transferring the Fano property from the special fiber $\overline{Y_\Gamma}$ to $\mathfrak{X}_g$. Recall from Propositions \ref{prop:family is flat} and \ref{prop:family is proj} that the morphism $\underline{\pi_\Gamma} : \overline{X_\Gamma} \to \Cc^{E(\Gamma)}$ is flat and projective, and $\overline{X_\Gamma}$ is Cohen-Macaulay. The fact that the family is Cohen-Macaulay ensures that the canonical sheaves of the fibers are specializations of the relative canonical sheaf $\omega_{\overline{X_\Gamma}/\Cc^{E(\Gamma)}}$ by \cite[Theorem 3.6.1]{Conrad}. Flatness provides us with the fact that $\mathfrak{X}_g$ is Gorenstein, as the special fiber is Gorenstein. Finally, as the family is projective (hence proper) and the anti-canonical line bundle on the special fiber is ample, there is an open neighborhood $U$ of $0\in \Cc^{E(\Gamma)}$ where $\omega_{\overline{X_\Gamma}_t}^{-1}$ is ample for all $t \in U$ by \cite[Theorem 1.2.17]{LazarsfeldI}. In particular, we can take $t \in (\Cc^*)^{E(\Gamma)}$ and see that $\mathfrak{X}_g$ is Fano. This proves our main result, Theorem \ref{thm-main}, from the introduction.

\begin{theorem}
\label{thm:compactification is GF}
For $t \in (\Cc^*)^{E(\Gamma)}$, $\underline{\pi_\Gamma}^{-1}(t) \simeq \mathfrak{X}_g$ is a Gorenstein Fano projective variety.
\end{theorem}

\nocite{*}
\bibliographystyle{alpha}
\bibliography{Fano}

\begin{thebibliography}{HMSV09}

\bibitem[BLR19]{BLR}
Indranil Biswas, Sean Lawton, and Daniel Ramras.
\newblock Wonderful compactification of character varieties.
\newblock {\em Pacific J. Math.}, 302(2):413--435, 2019.
\newblock With an appendix by Arlo Caine and Samuel Evens.

\bibitem[Con00]{Conrad}
Brian Conrad.
\newblock {\em Grothendieck duality and base change}, volume 1750 of {\em
  Lecture Notes in Mathematics}.
\newblock Springer-Verlag, Berlin, 2000.

\bibitem[Gol88]{Goldman}
William~M. Goldman.
\newblock Geometric structures on manifolds and varieties of representations.
\newblock In {\em Geometry of group representations ({B}oulder, {CO}, 1987)},
  volume~74 of {\em Contemp. Math.}, pages 169--198. Amer. Math. Soc.,
  Providence, RI, 1988.

\bibitem[HMSV09]{HMSV}
Benjamin Howard, John Millson, Andrew Snowden, and Ravi Vakil.
\newblock The equations for the moduli space of {$n$} points on the line.
\newblock {\em Duke Math. J.}, 146(2):175--226, 2009.

\bibitem[KM]{Kaveh-Manon-TVB}
Kiumars Kaveh and Christopher Manon.
\newblock Toric flat families, valuations, and tropical geometry over the
  semifield of piecewise linear functions.
\newblock arXiv:1907.00543 [math.AG].

\bibitem[KM19]{Kaveh-Manon-NOK}
Kiumars Kaveh and Christopher Manon.
\newblock Khovanskii bases, higher rank valuations, and tropical geometry.
\newblock {\em SIAM J. Appl. Algebra Geom.}, 3(2):292--336, 2019.

\bibitem[Laz04]{LazarsfeldI}
Robert Lazarsfeld.
\newblock {\em Positivity in algebraic geometry. {II}}, volume~49 of {\em
  Ergebnisse der Mathematik und ihrer Grenzgebiete. 3. Folge. A Series of
  Modern Surveys in Mathematics [Results in Mathematics and Related Areas. 3rd
  Series. A Series of Modern Surveys in Mathematics]}.
\newblock Springer-Verlag, Berlin, 2004.
\newblock Positivity for vector bundles, and multiplier ideals.

\bibitem[LM16]{Lawton-Manon}
Sean Lawton and Christopher Manon.
\newblock Character varieties of free groups are {G}orenstein but not always
  factorial.
\newblock {\em J. Algebra}, 456:278--293, 2016.

\bibitem[Man10]{Manon-Trees}
Christopher Manon.
\newblock Presentations of semigroup algebras of weighted trees.
\newblock {\em J. Algebraic Combin.}, 31(4):467--489, 2010.

\bibitem[Man18a]{Manon-ACB}
Christopher Manon.
\newblock The algebra of conformal blocks.
\newblock {\em J. Eur. Math. Soc. (JEMS)}, 20(11):2685--2715, 2018.

\bibitem[Man18b]{Manon-Outer}
Christopher Manon.
\newblock Toric geometry of {$SL_2(\Bbb{C})$} free group character varieties
  from outer space.
\newblock {\em Canad. J. Math.}, 70(2):354--399, 2018.

\bibitem[Sim94a]{Simpson1}
Carlos~T. Simpson.
\newblock Moduli of representations of the fundamental group of a smooth
  projective variety. {I}.
\newblock {\em Inst. Hautes \'{E}tudes Sci. Publ. Math.}, (79):47--129, 1994.

\bibitem[Sim94b]{Simpson2}
Carlos~T. Simpson.
\newblock Moduli of representations of the fundamental group of a smooth
  projective variety. {II}.
\newblock {\em Inst. Hautes \'{E}tudes Sci. Publ. Math.}, (80):5--79 (1995),
  1994.

\bibitem[Sim16]{Simpson}
Carlos Simpson.
\newblock The dual boundary complex of the {$SL_2$} character variety of a
  punctured sphere.
\newblock {\em Ann. Fac. Sci. Toulouse Math. (6)}, 25(2-3):317--361, 2016.

\bibitem[Vog89]{Vogt}
H.~Vogt.
\newblock Sur les invariants fondamentaux des \'{e}quations diff\'{e}rentielles
  lin\'{e}aires du second ordre.
\newblock {\em Ann. Sci. \'{E}cole Norm. Sup. (3)}, 6:3--71, 1889.

\bibitem[Wha20]{Whang}
Junho~Peter Whang.
\newblock Global geometry on moduli of local systems for surfaces with
  boundary.
\newblock {\em Compos. Math.}, 156(8):1517--1559, 2020.

\end{thebibliography}

\end{document}